\documentclass[11pt,reqno]{amsart}
\usepackage[margin=1.05in]{geometry}
\usepackage[T1]{fontenc}
\usepackage[utf8]{inputenc}
\usepackage{lmodern,amsmath,amssymb,booktabs,array,microtype,needspace}
\usepackage{float}
\usepackage[hidelinks]{hyperref}
\hypersetup{pdfauthor={Enkai Zhang},pdftitle={Exact diameter compression for four-point sumset cardinalities}}
\newtheorem{theorem}{Theorem}[section]
\newtheorem{lemma}[theorem]{Lemma}
\newtheorem{proposition}[theorem]{Proposition}
\newtheorem{corollary}[theorem]{Corollary}
\theoremstyle{definition}
\theoremstyle{remark}

\newcommand{\Z}{\mathbb Z}

\newcommand{\calR}{\mathcal R}
\newcommand{\ndiam}{\operatorname{ndiam}}

\title[Four-point sumset compression]{Exact diameter compression for four-point sumset cardinalities}
\author{Enkai Zhang}
\address{University of Toronto Scarborough, Toronto, Ontario, Canada}
\email{ek.zhang@mail.utoronto.ca}
\date{}
\begin{document}
\begin{abstract}
For each $h\ge2$, we determine the smallest integer interval realizing
every cardinality of an $h$-fold sumset of a four-element set of integers.
Write $\nu(h,4)$ for the least such diameter. In Problem~3 of
Nathanson's \emph{Compression and complexity for sumset sizes in additive
number theory}, the interval is $[0,N-1]\cap\mathbb Z$, so
$N(h,4)=\nu(h,4)+1$. We obtain $N(h,4)=\binom{h+2}{2}+(h\bmod2)+1$
for $2\le h\le10$ and $N(h,4)=(h-2)^2$ for $h\ge11$.
Put $d_h=(h-3)(h-1)$. For $h\ge6$, every four-set with exactly
$\binom{h+3}{3}-17$ distinct sums has normalized diameter in
$\{d_h,d_h+1,d_h+2\}$, and the lower endpoint is attained.
We prove the rigidity statement by classifying the relation lattice
at defect seventeen for $h\ge9$ and enumerating three smaller orders.
A lattice index bound and explicit short realizations give the general
upper bound for $h\ge17$. Exact finite searches cover the smaller orders.
\end{abstract}
\maketitle

\section{Introduction}\label{sec:intro}
For a finite set $A\subset\Z$ and an integer $h\ge1$, write
\[
 hA=\{a_1+\cdots+a_h:a_1,\ldots,a_h\in A\}.
\]
Repetition is allowed. We study the cardinality spectrum
\[
 \calR(h,4)=\{|hA|:A\subset\Z,\ |A|=4\}
\]
and the smallest integer interval that realizes all its values.
Following the diameter convention in \cite[Section 10.2]{Shin},
put
\begin{equation}\label{eq:nu}
 \nu(h,4)=\min\{D\in\Z_{\ge0}:
 \calR(h,4)=\{|hA|:A\subset[0,D]\cap\Z,\ |A|=4\}\}.
\end{equation}
An interval of diameter $D$ contains $D+1$ integer positions.
The definition concerns the number of distinct sums. Two sets
realizing the same cardinality may have different labeled
additive relations.

Nathanson \cite{Nathanson,NathansonCompression} and Shin
\cite{Shin} studied these spectra and their realization in
finite intervals.
In Problem~3 of \emph{Compression and complexity for sumset sizes}
\cite{NathansonCompression}, the interval is $[0,N-1]\cap\Z$.
Thus $N(h,4)=\nu(h,4)+1$, and Theorem~\ref{thm:compression}
answers the four-point case: $N(h,4)=\binom{h+2}{2}+(h\bmod2)+1$
for $2\le h\le10$, and
$N(h,4)=(h-3)(h-1)+1=(h-2)^2$ for $h\ge11$.

\Needspace{7\baselineskip}
\begin{theorem}[Exact compression]\label{thm:compression}
For every integer $h\ge2$,
\begin{equation}\label{eq:exactnu}
 \nu(h,4)=
 \begin{cases}
 \displaystyle\binom{h+2}{2}+(h\bmod2),&2\le h\le10,\\[3pt]
 (h-3)(h-1),&h\ge11.
 \end{cases}
\end{equation}
\end{theorem}

The formula for $h\ge17$ follows from a uniform proof. For
$11\le h\le16$ Theorem~\ref{thm:rigidity} gives the lower bound,
and the finite calculation in Section~\ref{sec:computation} gives the
upper bound. For $2\le h\le10$ the finite calculation gives both bounds.

For example, at $h=11$ the set
\begin{equation}\label{eq:counterexample}
 A=\{0,7,17,80\}
\end{equation}
has $347$ distinct sums. Theorem~\ref{thm:rigidity} below
shows that every four-set with this cardinality needs
diameter at least $80$. The maximal cardinality $364$ already
has a representative of diameter $79$
\cite[Theorem 10.3]{Shin}.

\Needspace{4\baselineskip}
Put $M_j=\binom{j+3}{3}$ for $j\ge0$, the number of
multisets of size $j$ on four elements, and call
$\Delta_h(A)=M_h-|hA|$ the \emph{defect}.
The lower bound in the second branch of \eqref{eq:exactnu}
is supplied by a single near-maximal cardinality.
For an attainable defect $\delta$, let
\[
 \eta_\delta(h)=\min\{\operatorname{ndiam}(A):
             A\subset\Z,\ |A|=4,\ \Delta_h(A)=\delta\},
\]
where $\operatorname{ndiam}(A)$ is the diameter after translation
and division by the gcd of all differences.
These operations preserve every sumset cardinality.
The notation follows \cite[Section 10.3]{Shin}, where $\eta_2$
is determined.

\Needspace{15\baselineskip}
\begin{theorem}[Defect-seventeen rigidity]\label{thm:rigidity}
Let $h\ge6$ and put $d_h=(h-3)(h-1)$.
Every integer four-set $A$ with $|hA|=M_h-17$ satisfies
\[
 \operatorname{ndiam}(A)\in\{d_h,d_h+1,d_h+2\}.
\]
Moreover, $\eta_{17}(h)=d_h$. The minimum is attained by
\begin{equation}\label{eq:uniform}
 A_h=
 \begin{cases}
 \{0,h-4,2h-5,d_h\},&h\not\equiv1\pmod3,\\
 \{0,h-3,2h-3,d_h\},&h\equiv1\pmod3.
 \end{cases}
\end{equation}
\end{theorem}

For $6\le h\le8$ the product bound reduces the rigidity statement
to three finite enumerations (Section~\ref{sec:rigidity}).
For all $h\ge9$ an explicit argument applies. At defect seventeen there are three active relation
directions, fifteen pair fibers and a unique triple fiber.
The latter has pairwise coordinate minima of sizes $2,1,1$.
Two support patterns and their exponent assignments determine
the restriction to three consecutive diameters.

For the upper bound in Theorem~\ref{thm:compression}, we may
replace a set by any other four-set with the same sumset
cardinality. After treating active rank at most one,
we show that a larger representative has only two or three
active directions. Their indices and degrees leave only
sums of basic defects and the four exceptional
defects $17,23,26,29$. We construct sufficiently short
representatives for all of them. Lemma~\ref{lem:replacement}
states precisely how this replacement proves an upper bound
for $\nu$.

\begin{corollary}\label{cor:asymptotic}
We have $\displaystyle\lim_{h\to\infty}\nu(h,4)/h^2=1$.
\end{corollary}
\begin{proof}
Divide the second branch of \eqref{eq:exactnu} by $h^2$.
\end{proof}

\subsection{Relation to previous work}
The relation-lattice framework, low-rank cardinality formula and
product diameter bound used here are established in
\cite{OBryant,Shin}. We recall their elementary proofs.
Shin's Theorem 10.3 determines the least diameter for the
maximal cardinality $M_h$, equivalently for a four-mark
$B_h$-ruler:
\begin{equation}\label{eq:ruler}
 D_h^{\max}=\binom{h+2}{2}+(h\bmod2).
\end{equation}
His Corollary 10.5 gives
\begin{equation}\label{eq:shinbounds}
 D_h^{\max}\le\nu(h,4)\le\max\{h^2,D_h^{\max}\}.
\end{equation}
More generally, Shin's Corollary~7.4 gives
$N(h,k)=\Theta_k(h^{k-2})$ for fixed $k$.
Nathanson \cite{NathansonExplicit} and Chen and Yang \cite{ChenYang}
compute sizes for explicit families. Rajagopal \cite{Rajagopal}
studies gaps in the spectrum, and Jung \cite{Jung} determines
the part between $4h+2$ and $6h-4$ for four-element sets.
Senger \cite{Senger} explains triangular gaps among the most frequent
sizes for four-element sets.
The present result determines the exact universal diameter
and the defect-seventeen minimum and rigidity.

The question in \cite[Section 10.2]{Shin} asks whether
$\nu(h,4)=D_h^{\max}$ for every $h\ge2$.
Theorem~\ref{thm:compression} identifies the first failure and
all subsequent values.

\subsection{Proof organization}
Section~\ref{sec:coverage} gives the relation and fiber tools.
Sections~\ref{sec:small}--\ref{sec:family} prove the
defect-seventeen lower bound and rigidity.
Sections~\ref{sec:large}--\ref{sec:special} prove the uniform
upper bound for $h\ge17$.
Section~\ref{sec:computation} gives the finite spectrum calculation.

\section{Relations and representation fibers}\label{sec:coverage}
For an integer set with at least two elements, let
\[
 g(A)=\gcd\{a-a':a,a'\in A\},\qquad
 \ndiam(A)=\frac{\max A-\min A}{g(A)}.
\]
The set is \emph{primitive} if $g(A)=1$.
Translation and division by $g(A)$ preserve every sumset cardinality.
Thus a minimum-diameter representative can be taken in the form
\begin{equation}\label{eq:normalized}
 A=\{0,a,b,D\},\qquad
 0<a<b<D,\qquad \gcd(a,b,D)=1.
\end{equation}

The homogeneous relation lattice of this ordered set is
\[
 L_A=\left\{z\in\Z^4:
 \sum_{i=0}^3z_i=0,\quad az_1+bz_2+Dz_3=0\right\}.
\]
For $z\in L_A$, write $z=z^+-z^-$ with disjoint nonnegative
parts, and define its \emph{degree} by
\[
\rho(z)=\sum_i z_i^+=\sum_i z_i^-.
\]
For any subset $I\subseteq\{0,1,2,3\}$,
\begin{equation}\label{eq:subsetmass}
 \left|\sum_{i\in I}z_i\right|\le\rho(z).
\end{equation}
Both the positive and negative mass are $\rho(z)$, so the sum
over a subset lies between $-\rho(z)$ and $\rho(z)$.
A relation is \emph{active at degree $h$} if it is nonzero and
$\rho(z)\le h$. The \emph{active rank} is the real dimension of
the span of the active relations; it is at most two.
An \emph{unsigned active direction} is an active relation up
to sign. Whenever twice the least relation degree exceeds $h$,
every active relation is primitive and there are no active
proper multiples.

Let
\[
 U_h=\{t\in\Z_{\ge0}^4:t_0+t_1+t_2+t_3=h\}.
\]
The map $t\mapsto t\cdot(0,a,b,D)$ sends $U_h$ onto $hA$.
The inverse image of a sum will be called its \emph{representation
fiber}. Any difference of two distinct points in a fiber is active. Conversely,
if $\rho(z)\le h$, the pairs
\[
 \gamma+z^-,\quad \gamma+z^+,\qquad
 \gamma\in U_{h-\rho(z)},
\]
are exactly the pairs of degree-$h$ representations differing by $z$.

The following form of the four-point product bound
\cite[Lemma 3.1]{Shin} makes the lattice index explicit.

\begin{lemma}[Product diameter bound]\label{lem:diameter}
If $u,v\in L_A$ are linearly independent, then
\begin{equation}\label{eq:product_index}
 [L_A:\Z u+\Z v]\,D\le \rho(u)\rho(v).
\end{equation}
Consequently active rank two implies $D\le h^2$.
\end{lemma}
\begin{proof}
Delete coordinate zero and denote the projected vectors by $u',v'$.
Their cross product is an integer multiple of $(a,b,D)$.
The absolute multiplier is $[L_A:\Z u+\Z v]$.
Indeed, a basis of the kernel of the primitive linear form
$(a,b,D)$, together with an integer vector on which the form
is one, is a basis of $\Z^3$. Its determinant is one up to
sign. Changing the kernel basis to $u',v'$ multiplies the
cross product by the determinant of that change, namely
the lattice index.

For any two coordinates $(x,y)$ of a zero-sum relation of degree $r$,
\[
 |x|\le r,\qquad |y|\le r,\qquad |x+y|\le r.
\]
Indeed the absolute sum over any subset of coordinates is at most
the total positive mass. The unit hexagon described by these three
inequalities has vertices
\[
 (1,0),(1,-1),(0,-1),(-1,0),(-1,1),(0,1).
\]
The absolute determinant of any two vertices is at most one.
By convexity in each argument, this remains true for any two points
of the hexagon. Hence every two-coordinate minor of $u',v'$ has
absolute value at most $\rho(u)\rho(v)$.
The last component of their cross product has absolute value
$[L_A:\Z u+\Z v]D$, proving the bound.
\end{proof}

\begin{lemma}[Replacing a realization]\label{lem:replacement}
Fix $h\ge2$ and $D_0\ge0$. Suppose that every integer four-set
$A$ with $\operatorname{ndiam}(A)>D_0$ has an integer four-set
$B$ satisfying
\[
 |hB|=|hA|,\qquad \operatorname{ndiam}(B)\le D_0.
\]
Then $\nu(h,4)\le D_0$.
\end{lemma}
\begin{proof}
Take any attainable cardinality and a normalized realization.
If its diameter is at most $D_0$, it already lies in the
required interval. Otherwise, use the stipulated set $B$ and
normalize it. This realizes the same cardinality in
$[0,D_0]\cap\Z$. Since the cardinality was arbitrary, every
element of $\calR(h,4)$ occurs there.
\end{proof}

The set $B$ in this lemma can have different labeled
equalities from $A$. Our upper-bound proof classifies the
possible cardinalities of large representatives and then
constructs another representative of each size.

For a finite set $S$ of zero-sum integer vectors, define its
\emph{translation height} by
\[
 H(S)=-\sum_i\min_{s\in S}s_i.
\]
This is the least common coordinate sum after translating all
points of $S$ into the nonnegative orthant. A translate occurs
in $U_h$ precisely when $H(S)\le h$.

\begin{lemma}[Fibers with few active directions]\label{lem:fiber_tools}
Let $r$ be the least nonzero relation degree. Suppose $2r>h$
and there are at most four unsigned active directions.
Every representation fiber then has at most three points.
If $E$ is the number of unordered collision edges and $T$
the number of triple fibers, then
\begin{equation}\label{eq:edge_count}
 \Delta_h(A)=E-T,\qquad
 E=\sum_{\text{unsigned active }z}M_{h-\rho(z)}.
\end{equation}
\end{lemma}
\begin{proof}
Every active vector is primitive, since a proper divisor would
have degree at most $h/2<r$. Four points in a fiber have six
pair differences, so two agree up to sign.
If their edges share a vertex, an active multiple $2u$ results,
contrary to $2\rho(u)>h$.
If the edges are disjoint, the points form a parallelogram
of shape $\{0,u,v,u+v\}$. Coordinatewise,
\[
 \min(0,u_i,v_i,u_i+v_i)=\min(0,u_i)+\min(0,v_i),
\]
so its height is $\rho(u)+\rho(v)\ge2r>h$, also impossible.
Thus each nontrivial fiber is a pair or a triple.
The padding description counts the edges of each direction.
A pair contributes one edge and one unit of defect, whereas
a triple contributes three edges and two units.
This gives \eqref{eq:edge_count}.
\end{proof}

A \emph{triangle type} is a set of three unsigned directions
occurring as pair differences of a three-point fiber.
For a fixed type there are two orientations, represented up to
translation by $\{0,u,v\}$ and $\{0,-u,-v\}$.
Their heights satisfy
\begin{equation}\label{eq:triangle_height_sum}
 H+H'=\rho(u)+\rho(v)+\rho(u-v).
\end{equation}
To verify this, apply
$|x-y|+|y-z|+|z-x|=2(\max\{x,y,z\}-\min\{x,y,z\})$
in each coordinate. In particular, both orientations cannot
occur in $U_h$ if the right side exceeds $2h$.

Under the hypotheses of Lemma~\ref{lem:fiber_tools}, a triangle
orientation of height $H\le h$ gives exactly $M_{h-H}$ triple
fibers. Different paddings give different fibers: otherwise
two translated three-point sets coincide, and summing their
vertices gives $3(\gamma-\gamma')=0$.
This counting rule will be used both for small defects and
for the explicit upper-bound constructions.

The next formula is the rank-zero and rank-one cardinality
formula from the relation-lattice framework; see
\cite[Theorem 2]{OBryant}.

\begin{lemma}[Rank at most one]\label{lem:rankone}
If the active rank is zero, then $|hA|=M_h$. If the active rank is
one, there is an integer $r$ with $2\le r\le h$ such that
\[
 |hA|=M_h-M_{h-r}.
\]
\end{lemma}
\begin{proof}
Rank zero means that the representation map is injective.
For rank one, let $z$ be the primitive integer generator of the active
line. It belongs to $L_A$, and its degree $r$ is at most $h$.
Since the four elements are distinct, $r\ge2$.

Every fiber lies on a line parallel to $z$. Its integer points form
a consecutive chain: the parameters for which all four coordinates
are nonnegative form an interval. A fiber with $m$ representations
therefore has exactly $m-1$ consecutive edges, each with difference
$z$ up to sign. The displayed padding description counts all such
edges by $|U_{h-r}|=M_{h-r}$. Subtracting one per edge from the
$M_h$ representation points gives the result.
\end{proof}

\section{The structure of small defects}\label{sec:small}
This section explains why a small number of lost sum values leaves
very little freedom in the representation fibers. Throughout,
$L=L_A$ is the full rank-two homogeneous relation lattice, and
$r=\min\{\rho(z):z\in L\setminus\{0\}\}$.
Under the hypotheses below every active relation is primitive,
so different multiples do not have to be counted separately.

\begin{lemma}\label{lem:short}
If $h\ge7$ and $0<\Delta_h(A)\le19$, then $r\ge h-2$,
every active relation is primitive, and there are at most four
unsigned active directions.
\end{lemma}
\begin{proof}
Positive defect supplies two distinct points in a fiber; their
difference has degree at most $h$. Thus $r\le h$ (if $r>h$, the
representation map is injective). For a primitive relation of degree $r$, the padded edges
$\gamma+z^-\leftrightarrow\gamma+z^+$ form disjoint chains in the
representation fibers. Even if a fiber contains more than one
chain, it loses at least as many values as there are chain edges.
Consequently
\[
 \Delta_h(A)\ge M_{h-r}.
\]
Since $M_3=20$, we have $r\ge h-2$.
A proper divisor of an active relation would have degree at most
$h/2<r$, so every active relation is primitive.

The active relations, together with zero, inject into $L/3L$.
Otherwise distinct vectors $u,v$ in this set would satisfy
\[
 0\ne (u-v)/3\in L,\qquad
 \rho((u-v)/3)\le 2h/3<h-2\le r.
\]
This contradicts the definition of $r$.
The group $L/3L$ has nine elements because $L$ has rank two.
There are therefore at most eight nonzero active vectors, or four
unsigned directions.
\end{proof}

\begin{lemma}\label{lem:fibers}
Under the hypotheses of Lemma~\ref{lem:short}, every representation
fiber has at most three points. There is at most one unsigned
active direction of degree $h-2$.
\end{lemma}
\begin{proof}
The fiber bound follows from Lemma~\ref{lem:fiber_tools},
since $2r\ge2(h-2)>h$.

Suppose two directions have degree $h-2$.
They contribute ten padded edges each. The graph formed by these
two directions is a forest in each fiber: a fiber has at most
three points, and a triangle uses three different directions.
The twenty edges therefore force defect at least twenty, a
contradiction.
\end{proof}

\begin{lemma}\label{lem:triangle}
Under the hypotheses of Lemma~\ref{lem:short}, at most one triangle
type and one orientation occur in degree $h$.
\end{lemma}
\begin{proof}
By \eqref{eq:triangle_height_sum}, the heights of the two
orientations of one type satisfy
\[
 H+H'=\rho(u)+\rho(v)+\rho(u-v)\ge3(h-2)>2h.
\]
Thus both orientations cannot occur.

Two different triangle types, if present, share exactly two
directions: each uses three, and there are at most four in all.
Represent the shared directions by independent vectors $u,v$.
The other directions are $u+v$ and $u-v$, up to sign.
The four triangles obtained by deleting one corner of
$\{0,u,v,u+v\}$ represent the two orientations of these two types.
Put $R=\rho(u)+\rho(v)$. Deleting the indicated corner decreases
the parallelogram height by
\begin{align*}
 s_0&=\sum_i\min((u_i)_+,(v_i)_+),&
 s_u&=\sum_i\min((-u_i)_+,(v_i)_+),\\
 s_v&=\sum_i\min((u_i)_+,(-v_i)_+),&
 s_{u+v}&=\sum_i\min((-u_i)_+,(-v_i)_+).
\end{align*}
Here $x_+=\max(x,0)$.
These formulas follow by taking the coordinatewise minimum before
and after deleting the corner. The remaining triangle has height
$R-s_j$.

Opposite deleted corners give the same type. One triangle from
each type therefore corresponds to adjacent deleted corners.
Their two savings sum to at most $\rho(u)$ or $\rho(v)$, hence
at most $h$. For example,
$s_0+s_u\le\sum_i(v_i)_+=\rho(v)$; the other cases follow by
exchanging $u,v$ or changing their signs.
If both triangles fit in degree $h$, each saving is at least
$R-h$, giving $2R\le3h$.
But at most one of the shared directions has degree $h-2$, by
Lemma~\ref{lem:fibers}. Therefore
\[
 R\ge(h-2)+(h-1)=2h-3,\qquad 2R\ge4h-6>3h
\]
for $h\ge7$, a contradiction.
\end{proof}

\begin{proposition}[Small-defect counting]\label{prop:small}
Let $h\ge7$ and $0\le\Delta_h(A)\le19$.
All active directions have degrees $h-2,h-1,h$, with respective
counts $\varepsilon,b,c$, where
\[
 \varepsilon\in\{0,1\},\qquad \varepsilon+b+c\le4.
\]
Every nontrivial representation fiber is a pair or a triple.
If $T$ is the number of triple fibers, then
\begin{equation}\label{eq:small}
 \Delta_h(A)=10\varepsilon+4b+c-T,\qquad T\in\{0,1,4\}.
\end{equation}
If $T>0$, at least three directions occur; if $T=4$, at least
three of the directions have degree at most $h-1$.
\end{proposition}
\begin{proof}
Zero defect means that there are no active relations.
Otherwise the preceding lemmas apply. A triangle has height at
least $h-1$, since it uses three directions and at most one can
have smaller degree. Its height is thus $h-1$ or $h$ when it
occurs. The number of nonnegative common paddings is respectively
$M_1=4$ or $M_0=1$.

The padding rule following Lemma~\ref{lem:fiber_tools}, together
with Lemma~\ref{lem:triangle}, gives the possibilities for $T$.

The total number of collision edges is
$10\varepsilon+4b+c$. A pair fiber contributes one edge and one
unit of defect. A triple contributes three edges and two units.
Subtracting $T$ gives \eqref{eq:small}.
The last assertions follow from the three edges of a triangle
and from its height when $T=4$.
\end{proof}

At defect seventeen, \eqref{eq:small} forces a single degree pattern.

\section{The diameter restriction at defect seventeen}\label{sec:rigidity}
We now prove the first assertion of Theorem~\ref{thm:rigidity}.
For $h=6,7,8$ the following finite argument proves both conclusions.
A defect of $17$ is neither $0$ nor $M_j$ for any $j\ge0$.
By Lemma~\ref{lem:rankone} the active rank is two, so
Lemma~\ref{lem:diameter} bounds the normalized diameter by $h^2$.
For each $3\le D\le h^2$, take all $1\le a<b<D$ with $\gcd(a,b,D)=1$.
Starting from $S_0=\{0\}$ use
$S_{i+1}=S_i\cup(a+S_i)\cup(b+S_i)\cup(D+S_i)$ until $i+1=h$.
The following are all nonzero counts for $|S_h|=M_h-17$;
reflection is counted separately.
\[
\begin{array}{ccll}
 h&D&\text{counts in order of }D&A_h\\
6&15,16,17&4,8,24&\{0,2,7,15\}\\
7&24,25,26&4,14,14&\{0,4,11,24\}\\
8&35,36,37&20,12,38&\{0,4,11,35\}
\end{array}
\]
The final column gives the member of \eqref{eq:uniform} attaining the first
diameter in each row. The same search at $h=5$ gives only $D=9,10$
with counts $2,8$, whereas $d_5=8$. The finite computation can be
repeated with the program in Appendix~\ref{app:reproduction}.

For the remainder of this section assume $h\ge9$ and $\Delta_h(A)=17$.
Proposition~\ref{prop:small} leaves only one possible set of degrees.
The remaining task is to recover the set from the exponent
vectors of its triple fiber.

\begin{lemma}\label{lem:seventeen}
There are exactly three unsigned active relations, of degrees
$h-2,h-1,h-1$. There are fifteen pair fibers and one triple fiber.
\end{lemma}
\begin{proof}
If $\varepsilon=0$ in \eqref{eq:small}, the defect is at most
$4(b+c)\le16$. Thus $\varepsilon=1$, and
\[
 17=10+4b+c-T,\qquad b+c\le3,\qquad T\in\{0,1,4\}.
\]
For $T=0,1,4$ this requires respectively $4b+c=7,8,11$.
The only solution with $b,c\ge0$ and $b+c\le3$ is
$(b,c,T)=(2,0,1)$. The eighteen edges therefore comprise three
edges of the triple and fifteen edges of pair fibers.
\end{proof}

Let $P,Q,R\in U_h$ be the triple fiber, labeled so that
\[
 \rho(Q-P)=h-2,\qquad \rho(R-P)=\rho(R-Q)=h-1.
\]
Its two edges $u=Q-P$ and $v=R-P$ determine the whole lattice,
not just the relations visible at degree $h$.

\begin{lemma}\label{lem:basis}
The vectors $u,v$ form an integer basis of $L_A$.
\end{lemma}
\begin{proof}
They are independent; otherwise the fiber would contain an
equally spaced triple and an active multiple of an active vector.
By Lemma~\ref{lem:seventeen}, the full active set is
$\{\pm u,\pm v,\pm(v-u)\}$.
If the subgroup generated by $u,v$ had index greater than one,
its half-open fundamental parallelogram would contain a nonzero
point
\[
 w=\alpha u+\beta v\in L_A,\qquad 0\le\alpha,\beta<1.
\]
Replace $w$ by $u+v-w$ if $\alpha+\beta>1$.
Denote its coefficients again by $\alpha,\beta$; they now
satisfy $\alpha,\beta\ge0$ and $\alpha+\beta\le1$.
We obtain a nonvertex lattice point of the triangle
$\operatorname{conv}\{0,u,v\}$.
It cannot lie on a radial edge, since $u,v$ are primitive.
Its direction is therefore strictly inside their positive cone.
Convexity of $\rho$ gives
\[
 \rho(w)\le\alpha\rho(u)+\beta\rho(v)
 =\alpha(h-2)+\beta(h-1)
 \le(h-1)(\alpha+\beta)\le h-1.
\]
None of the six listed active vectors has this direction, a
contradiction.
\end{proof}

For equal-degree nonnegative vectors,
\[
 \sum_i\min(P_i,Q_i)=h-\rho(P-Q).
\]
Thus the pairwise coordinate minima of $(P,Q)$, $(P,R)$ and
$(Q,R)$ have total degrees $2,1,1$.
The common minimum of all three vectors is zero.
Otherwise the triangle could be translated into degree at most
$h-1$ and would give at least four triple fibers in degree $h$,
contrary to Lemma~\ref{lem:seventeen}.
The supports of the three pairwise minima are therefore disjoint.

\subsection{The exponent assignments}
After permuting the four coordinates, there are two possible
base matrices, whose rows correspond to $P,Q,R$:
\begin{equation}\label{eq:bases}
 B_{\rm I}=
 \begin{pmatrix}2&1&0&0\\2&0&1&0\\0&1&1&0\end{pmatrix},
 \qquad
 B_{\rm II}=
 \begin{pmatrix}1&1&1&0\\1&1&0&1\\0&0&1&1\end{pmatrix}.
\end{equation}
In Type I, the degree-two minimum occupies one coordinate; in
Type II, it occupies two. The degree-one minima each occupy one
of the remaining coordinates.

In a shared coordinate, any exponent above the base value can
be added to only one of the two rows sharing that coordinate.
Adding to both would increase their prescribed minimum.
In Type I the fourth coordinate can occur in only one row:
all three pairwise minima have already used their full degree
in the first three coordinates.
This fourth coordinate must be positive. If it were zero in
every row, two independent homogeneous relations would be
supported on three distinct set elements. The homogeneous
relation space on three distinct elements has dimension one,
so this is impossible.

Put $n=h-3$, so $n\ge6$.
The extra row sums beyond either base matrix are $n,n,n+1$.
There are four extra-exponent positions, and each row must
receive at least one. Thus one row receives two, and the others
receive one each. In the row receiving two, assign $t$ to the
first position in coordinate order and the remaining row sum
to the second. All other extra entries are fixed by their row
sums.

The choices can be recorded without listing the full matrix each
time. A bit $0$ selects the first row of a pair and a bit $1$
selects the second. The pairs, in coordinate order, are
$PQ,PR,QR$ for Type I and $PQ,PQ,PR,QR$ for Type II.
Type I also records the row receiving the exclusive fourth
coordinate.
Requiring all three rows to receive an extra position gives the
following complete list:
\[
\begin{array}{ccl}
 \text{Type I}&P:&010,\ 101,\ 110,\ 111,\\
              &Q:&001,\ 010,\ 011,\ 101,\\
              &R:&000,\ 010,\ 100,\ 101;\\[2pt]
 \text{Type II}&&0010,\ 0101,\ 0110,\ 0111,\
                         1001,\ 1010,\ 1011,\ 1101 .
\end{array}
\]
There are twelve assignments of Type I and eight of Type II.
If a shared coordinate receives no increment, assigning that zero
to either incident row leaves $P,Q,R$ unchanged. Thus the descriptions
can coincide at an endpoint of the $t$-interval.

Exchanging $P,Q$ reduces Type I to six representatives.
For Type II, also exchange its first two coordinates.
The resulting three representatives, together with those of
Type I, are listed in Table~\ref{tab:forms}.
For example, the exchange of $P,Q$ sends a Type I code
$(e;b_0b_1b_2)$ to
$(\tau(e);(1-b_0)b_2b_1)$, where $\tau$ exchanges $P,Q$ and
fixes $R$. In Type II the two operations send the code to
$b_1b_0b_2b_3$ and
$(1-b_0)(1-b_1)b_3b_2$.
These operations preserve the unordered set of exponents up to coordinate
relabeling. Their orbits are shown below; in Type~I the first letter
gives the exclusive row. The orbits have sizes $2,2,2,2,2,2,2,4,2$
in the order displayed.
\[
\begin{array}{c l @ {\qquad} c l}
A&P010,Q101&F&R010,R101\\
B&P101,Q010&G&0010,1101\\
C&P110,Q001&H&0101,0110,1001,1010\\
D&P111,Q011&I&0111,1011\\
E&R000,R100& & \\
\end{array}
\]
All twenty codes therefore occur in disjoint orbits.

\begin{table}[H]
\centering
\caption{Representative exponent assignments.
The exclusive row is used only in Type I.
The parameter $t$ is an integer; its bounds incorporate the
positive fourth coordinate in Type I.}\label{tab:forms}
\begin{tabular}{ccccc}
\toprule
Label&Type&Exclusive row&Code&Range of $t$\\
\midrule
A&I&P&010&$0\le t\le n-1$\\
B&I&P&101&$0\le t\le n-1$\\
C&I&P&110&$0\le t\le n$\\
D&I&P&111&$0\le t\le n+1$\\
E&I&R&000&$0\le t\le n$\\
F&I&R&010&$0\le t\le n$\\
G&II&--&0010&$0\le t\le n$\\
H&II&--&0101&$0\le t\le n$\\
I&II&--&0111&$0\le t\le n+1$\\
\bottomrule
\end{tabular}
\end{table}

For direct substitution, the three rows for every representative are
{\small\setlength{\arraycolsep}{3pt}
\[
\begin{array}{c@{\quad}lll}
 &P&Q&R\\
A&(t+2,1,0,n-t)&(2,0,n+1,0)&(0,n+2,1,0)\\
B&(2,t+1,0,n-t)&(n+2,0,1,0)&(0,1,n+2,0)\\
C&(2,1,0,n)&(t+2,0,n-t+1,0)&(0,n+2,1,0)\\
D&(2,1,0,n)&(n+2,0,1,0)&(0,t+1,n-t+2,0)\\
E&(t+2,n-t+1,0,0)&(2,0,n+1,0)&(0,1,1,n+1)\\
F&(n+2,1,0,0)&(2,0,n+1,0)&(0,t+1,1,n-t+1)\\
G&(t+1,n-t+1,1,0)&(1,1,0,n+1)&(0,0,n+2,1)\\
H&(t+1,1,n-t+1,0)&(1,n+1,0,1)&(0,0,1,n+2)\\
I&(n+1,1,1,0)&(1,n+1,0,1)&(0,0,t+1,n-t+2)\\
\end{array}
\]}
Each row sums to $h=n+3$; the allowed values of $t$ are those of
Table~\ref{tab:forms}.

\subsection{Recovering the set}
For each assignment, delete coordinate zero from $u,v$ and form
\[
 (q_1,q_2,q_3)=u'\times v',\qquad q_0=0.
\]
Lemma~\ref{lem:basis} implies that this normal is primitive.
Indeed, a basis of the kernel of a primitive integer linear
form, together with a vector on which the form is one, is a
basis of $\Z^3$. Its determinant is one up to sign.
Consequently the $q_i$ are the elements of the normalized set up
to a common translation and sign, even though the coordinates
have been permuted. In particular,
\begin{equation}\label{eq:normaldiam}
 \operatorname{ndiam}(A)=\max_{0\le i\le3}q_i-\min_{0\le i\le3}q_i.
\end{equation}
The index is one, so the primitive normal gives the
normalized diameter directly, with no common factor to remove.

To see how an assignment gives its normal, consider case A.
The row choices in Table~\ref{tab:forms} give
\[
 P=(2+t,1,0,n-t),\quad
 Q=(2,0,n+1,0),\quad
 R=(0,n+2,1,0).
\]
Subtracting $P$ from $Q$ and $R$ and taking the projected
cross product gives row A below. The other assignments use
the same rule.
Writing $d=d_h=n(n+2)$, the nine normals are:
\begin{align*}
 {\rm A}:\quad &(n(t-n),\ (n+2)(t-n),\ -d-2),\\
 {\rm B}:\quad &((n+1)(n-t),\ t-n,\ -n-2-(n+1)t),\\
 {\rm C}:\quad &(n(t-n),\ -d,\ (n+1)t-d-2),\\
 {\rm D}:\quad &(n(n+1-t),\ -n(t+1),\ -n-2),\\
 {\rm E}:\quad &(d+1,\ (n+1)(n+1-t),\ n(n-t)-1),\\
 {\rm F}:\quad &((n+1)(n+1-t),\ n+1-t,\ -1-(n+1)t),\\
 {\rm G}:\quad &(-d-2,\ nt-n(n+1)-1,\ (n+2)(t-n)-1),\\
 {\rm H}:\quad &((n+1)t-d-2,\ -d-1,\ (n+1)(t-n)-1),\\
 {\rm I}:\quad &(-n-2,\ nt-d-1,\ nt-1).
\end{align*}

Each normal has a pair of entries separated by a constant in
$t$. For A--I, respectively, these separations are
\[
 q_0-q_3=d+2,\quad q_1-q_3=d+2,\quad q_0-q_2=d,\quad
 q_1-q_2=d,\quad q_1-q_0=d+1,
\]
\[
 q_1-q_3=d+2,\quad q_0-q_1=d+2,\quad
 q_0-q_2=d+1,\quad q_3-q_2=d.
\]
Thus every range in \eqref{eq:normaldiam} is at least $d$.
At the two endpoints of the corresponding $t$-interval in
Table~\ref{tab:forms}, the ranges minus $d$ are respectively
\[
 (2,2),\ (2,2),\ (2,0),\ (2,0),\ (1,2),\
 (2,2),\ (2,2),\ (2,1),\ (1,1).
\]
For example, row A, including $q_0=0$, gives
\[
 \begin{array}{c|c}
 t&(q_0,q_1,q_2,q_3)\\ \hline
 0&(0,-n^2,-d,-d-2)\\
 n-1&(0,-n,-n-2,-d-2).
 \end{array}
\]
At both endpoints the maximum is $0$ and the minimum is
$-d-2$, so the range is $d+2$. This gives the first pair
$(2,2)$ above. The remaining pairs follow by the same
substitution and comparison, using $n\ge6$.
The range of finitely many affine functions of $t$ is convex:
it equals $\max_{i,j}(q_i-q_j)$.
It is therefore no larger than the larger endpoint range.
Every range is at most $d+2$.
Since the range is integral, \eqref{eq:normaldiam} proves
the three-value restriction in Theorem~\ref{thm:rigidity}.
The next section attains its lower endpoint.

\section{A construction with constant defect seventeen}\label{sec:family}
It remains to attain the lower endpoint in Theorem~\ref{thm:rigidity}.
We give the construction together with its complete relation list.

\begin{proposition}\label{thm:family}
For every integer $h\ge9$, the set $A_h$ in \eqref{eq:uniform}
is primitive and satisfies
\[
 \operatorname{diam}(A_h)=(h-3)(h-1),\qquad |hA_h|=M_h-17.
\]
\end{proposition}

We use two complementary families.
Let
\[
 v=h-1,\qquad u=v-2,\qquad d=uv.
\]
For $v\ge8$, consider
\begin{align}
 A(v,s)&=\{0,u-s,u+v-s,uv\},
 &&s\in\{1,2\},\quad\gcd(v,s+2)=1,\label{eq:familyA}\\
 B(v)&=\{0,v-2,2v-1,v(v-2)\},
 &&\gcd(v-2,3)=1.\label{eq:familyB}
\end{align}
They are primitive: the gcd in \eqref{eq:familyA} is
$\gcd(u-s,v)=\gcd(v,s+2)$, while the gcd of the middle two entries
in \eqref{eq:familyB} is $\gcd(v-2,3)$.
All four entries are distinct and in increasing order.
If $3\nmid v$, the first family with $s=1$ is available; if
$3\mid v$, the second is available. These choices are exactly
\eqref{eq:uniform}.

\subsection{Complete short-relation lists}
We will use the following elementary lattice fact. If two projected
relations in $\Z^3$ have cross product equal to the primitive
normal vector $(a,b,d)$ up to sign, they form an integer basis of
its kernel. To see this, choose an integer vector $t$ with
$(a,b,d)\cdot t=1$. Adjoining $t$ gives a determinant-one basis of
$\Z^3$, and the coefficient of $t$ vanishes for every vector in
the kernel. Coordinate zero is recovered from homogeneity.

\begin{lemma}\label{lem:Alist}
For the set $A(v,s)$, all nonzero relations of degree at
most $h=v+1$ are
\[
 \pm R,\quad \pm S,\quad \pm(R+S),
\]
where
\[
 R=(1,-u,u,-1),\qquad S=(1-v,v-s,s,-1).
\]
Their degrees are $h-2,h-1,h-1$, respectively.
\end{lemma}
\begin{proof}
The displayed vectors are relations, and their projections have
cross product $-(u-s,u+v-s,uv)$. They therefore form a full lattice
basis. Their degrees follow directly from the signs of their
coordinates; in particular,
$R+S=(2-v,2-s,u+s,-2)$ has degree $v$ because $s\le2$.

Write a short relation as $mR+nS$. Coordinates one and two sum
to $vn$. The absolute sum over any subset of coordinates is at
most the degree, so $v|n|\le v+1$. Thus $n\in\{-1,0,1\}$.
After changing sign, it suffices to consider $n=0,1$.

For $n=0$, the degree is $|m|(v-1)$, which forces $|m|\le1$.
For $n=1$, coordinate two is $um+s$. Its absolute value is at
most $v+1$. If $|m|\ge2$, then
\[
 |um+s|\ge u|m|-s\ge2u-s
 =2v-4-s\ge2v-6\ge v+2>v+1,
\]
where $u=v-2$, $v\ge8$ and $1\le s\le2$.
Thus $m\in\{-1,0,1\}$. The case $m=-1$ has degree
$\rho(S-R)=2v-2-s>v+1$ and is excluded. The remaining cases
give $S$ and $R+S$.
\end{proof}

\begin{lemma}\label{lem:Blist}
For $B(v)$, all nonzero relations of degree at most $h=v+1$ are
\[
 \pm P,\quad \pm Q,\quad \pm(P+Q),
\]
where
\[
 P=(1-v,v,0,-1),\qquad Q=(2,1-v,u,-1).
\]
Their degrees are $h-1,h-1,h-2$, respectively.
\end{lemma}
\begin{proof}
Their projected cross product is $(u,2v-1,uv)$, so they are a
full lattice basis. The degrees follow from the coordinates,
including $P+Q=(3-v,1,u,-2)$.

For a short relation $mP+nQ$, coordinate two is $un$.
Since $2u>v+1$, we have $|n|\le1$. If $n=0$, then $|m|\le1$.
When $n=1$, coordinate one is $v(m-1)+1$, whose absolute bound
forces $m\in\{0,1,2\}$. At $m=2$, coordinate zero is $4-2v$,
with absolute value greater than $v+1$. Thus only $Q$ and
$P+Q$ remain. Sign symmetry treats $n=-1$.
\end{proof}

\subsection{Counting the representation fibers}
For a relation $z$ counted up to sign, the number of unordered
degree-$h$ representation pairs with difference $\pm z$ is
$M_{h-\rho(z)}$, by the padding description in
Section~\ref{sec:coverage}. Both families therefore have
\begin{equation}\label{eq:edges}
 M_2+2M_1=10+4+4=18
\end{equation}
collision edges. Counting edges alone is not enough: a fiber of
three points has three edges but decreases the number of distinct
sums by only two.

In the lattice basis of either lemma, the allowable differences
are
\[
 \pm(1,0),\quad\pm(0,1),\quad\pm(1,1).
\]
A fiber has at most three points. Indeed any triple of pairwise
allowable differences uses all three directions. Fixing an edge
with endpoints $(0,0),(1,0)$, its third point must be $(0,-1)$
or $(1,1)$; both cannot occur. Thus there are two triangle shapes
up to translation.

For Family~\eqref{eq:familyA}, representatives of the shapes are
$\{0,R,-S\}$ and $\{0,R,R+S\}$. The least coordinatewise
translation making each shape nonnegative is, respectively,
\[
 (0,v-s,s,1),\qquad (v-2,u,0,2).
\]
Their coordinate sums are $v+1=h$ and $2v-2>h$. Hence only the
first shape can appear at degree $h$, and it admits exactly one
translation at that degree. Its three exponent vectors are
\begin{equation}\label{eq:tripleA}
 (0,v-s,s,1),\quad
 (1,2-s,u+s,0),\quad
 (v-1,0,0,2).
\end{equation}
They all represent the sum $2d$.

For Family~\eqref{eq:familyB}, the triangle $\{0,P,P+Q\}$ has
least translation $(v-1,0,0,2)$ of degree $h$. Its three vectors
are
\begin{equation}\label{eq:tripleB}
 (v-1,0,0,2),\quad (0,v,0,1),\quad (2,1,u,0),
\end{equation}
again all representing $2d$. The other triangle,
$\{0,P,-Q\}$, has least translation $(v-1,0,u,1)$ of degree
$2v-2>h$.

Each family therefore has exactly one triple fiber and no larger
fiber. Of the $18$ edges in \eqref{eq:edges}, three belong to
that triple and the other $15$ belong to double fibers. The
defect is
\[
 15+(3-1)=17.
\]
This proves Proposition~\ref{thm:family}.

\section{Restrictions on large representatives}\label{sec:large}
We turn to the upper bound in Theorem~\ref{thm:compression}.
Throughout Sections~\ref{sec:large}--\ref{sec:special}, assume
$h\ge17$ and write $d_h=(h-3)(h-1)$.
We will classify the cardinalities that a representative of
diameter greater than $d_h$ can have, then realize those
cardinalities by other sets of diameter at most $d_h$.

\begin{lemma}\label{lem:large_ranktwo}
Let $A=\{0,a,b,D\}$ be normalized, with $D>d_h$ and active
rank two. Its least relation degree is at least $h-3$.
It has at most three unsigned active directions, and every
two independent active vectors form a basis of $L_A$.
\end{lemma}
\begin{proof}
Let $r$ be the least nonzero relation degree.
The active rank is two, so a shortest relation has an
independent active partner. If $r\le h-4$,
Lemma~\ref{lem:diameter} gives
\[
 D\le rh\le h(h-4)<d_h,
\]
a contradiction. Thus $r\ge h-3$.
The active vectors and zero inject into $L_A/3L_A$, as in
Lemma~\ref{lem:short}, since $2h/3<h-3$.
There are at most four unsigned directions. Every active
vector is primitive because $h/2<r$.

Suppose there are four directions.
A primitive lattice vector is nonzero modulo $2L_A$,
and its negative has the same residue.
The rank-two group $L_A/2L_A$ has only three nonzero classes,
so two representatives $p,q$ of different directions have
the same residue. They are independent.
Both $(p+q)/2$ and $(p-q)/2$ belong to $L_A$, so
the index of $\Z p+\Z q$ in $L_A$ is divisible by two.
The index factor in \eqref{eq:product_index} yields
\[
 2D\le\rho(p)\rho(q)\le h^2.
\]
But $d_h>h^2/2$ for $h\ge8$, contradicting $D>d_h$.
There are therefore at most three directions.

Finally, any independent active pair of index at least two
would give the same contradiction. Every such pair is a
full lattice basis.
\end{proof}

\begin{proposition}\label{prop:large_defects}
Under the hypotheses of Lemma~\ref{lem:large_ranktwo}, the
defect of $A$ is either a sum of at most three members of
$\{1,4,10,20\}$ or one of $17,23,26,29$.
\end{proposition}
\begin{proof}
If there are three directions, let two be $u,v$, a lattice
basis by Lemma~\ref{lem:large_ranktwo}, and write the third
as $pu+qv$. Its pairs with $u$ and $v$ have indices
$|q|$ and $|p|$. Both are one. The third direction is
therefore $u+v$ or $u-v$, up to sign.
There is one triangle type.

Write the ordered degree deficits as $j_1\ge j_2$ in the
two-direction case, or $j_1\ge j_2\ge j_3$ in the three-direction
case: the degrees are $h-j_1,h-j_2,h-j_3$.
They lie in $\{h-3,h-2,h-1,h\}$.
The product of the two smallest degrees must exceed $d_h$,
which is equivalent to
\[
 (4-j_1-j_2)h+j_1j_2-3>0.
\]
For $h\ge17$ and $0\le j_2\le j_1\le3$, the possible pairs are
\begin{equation}\label{eq:degree_pairs}
 (j_1,j_2)\in
 \{(0,0),(1,0),(1,1),(2,0),(2,1),(2,2),(3,0)\}.
\end{equation}
The pair $(3,1)$ gives equality and is excluded.
If $j_1+j_2\ge5$, the product is smaller than $d_h$.

There are no active multiples and $2(h-3)>h$, so
Lemma~\ref{lem:fiber_tools} applies.
With two directions no triangle is possible, and the
defect is $M_{j_1}+M_{j_2}$.
With three directions, \eqref{eq:degree_pairs} leaves
\[
 0\le j_3\le j_2\le j_1\le2,\qquad\hbox{or}\qquad(j_1,j_2,j_3)=(3,0,0).
\]
The height sum of the two triangle orientations is at least
$3(h-3)>2h$. Hence at most one orientation occurs.
Its height is at least the largest edge degree $h-j_3$.
The padding count and \eqref{eq:edge_count} give
\begin{equation}\label{eq:large_defect_count}
 \Delta_h(A)=M_{j_1}+M_{j_2}+M_{j_3}-T,\qquad
 T\in\{0\}\cup\{M_0,\ldots,M_{j_3}\}.
\end{equation}
The nonzero corrections are listed in
Table~\ref{tab:exceptional}.
\begin{table}[H]
\centering
\caption{Possible triple-fiber corrections for a large
representative, grouped by the degree deficits of its three active directions.}\label{tab:exceptional}
\begin{tabular}{ccc}
\toprule
$(j_1,j_2,j_3)$&Collision edges&Defects with $T>0$\\
\midrule
$(0,0,0)$&3&2\\
$(1,0,0)$&6&5\\
$(1,1,0)$&9&8\\
$(1,1,1)$&12&11,\ 8\\
$(2,0,0)$&12&11\\
$(2,1,0)$&15&14\\
$(2,1,1)$&18&17,\ 14\\
$(2,2,0)$&21&20\\
$(2,2,1)$&24&23,\ 20\\
$(2,2,2)$&30&29,\ 26,\ 20\\
$(3,0,0)$&22&21\\
\bottomrule
\end{tabular}
\end{table}
Every entry except $17,23,26,29$ is a sum of at most two
members of $\{M_0,M_1,M_2,M_3\}=\{1,4,10,20\}$.
When $T=0$, the defect is already a sum of three such members.
This proves the proposition.
\end{proof}

The low-rank cases require only a standard principal
representative.
\begin{lemma}\label{lem:principal_representative}
If the active rank is one and its primitive relation has
degree $r\le h-4$, its cardinality has a representative
of diameter less than $d_h$.
\end{lemma}
\begin{proof}
By Lemma~\ref{lem:rankone}, the defect is $M_{h-r}$.
Use the principal representative
\cite[Section 10.2, (191)--(192)]{Shin}:
\[
 C_{h,r}=\{0,1,(r-1)h+1,(r-1)h+r\}.
\]
Put $m=(r-1)h+1$.
A short relation satisfies
\[
 m(z_2+z_3)+z_1+(r-1)z_3=0.
\]
The coefficients of the last expression lie in $[0,r-1]$,
so its absolute value is at most
$(r-1)\rho(z)\le(r-1)h<m$.
It follows that $z_2+z_3=0$ and every short relation is a
multiple of $(r-1,-(r-1),-1,1)$.
Thus $C_{h,r}$ has the required rank-one cardinality.
Its diameter is
\[
 (r-1)h+r\le h^2-4h-4<d_h.
\]
\end{proof}

The remaining rank-one defects are $M_0,M_1,M_2,M_3$.
Together with defect zero, these are included in the
constructions of the next section.

\section{Uniform short realizations}\label{sec:short_realizations}
Proposition~\ref{prop:large_defects} reduces most of the upper
bound to realizing sums of members of $\{1,4,10,20\}$.
We give two constructions that handle all of these sums at once.
For convenience put $M_{-1}=0$ in this section; an index $-1$
will mean that the corresponding direction is inactive.

For one or two basic defects, we use two pairs of nearby
integers: $0,1$ and $q,q+m$. The gap $m$ fixes one relation
degree. Choosing $q$ then sets the degree of one further
relation while keeping all other relations above degree $h$.

\begin{lemma}[One or two basic defects]\label{lem:two_defects}
Let $h\ge17$ and $j_1,j_2\in\{-1,0,1,2,3\}$.
Put
\begin{equation}\label{eq:two_family}
 m=h-j_1-1,\qquad \ell=\lfloor m/2\rfloor,\qquad
 q=m(h-\ell)+\ell-j_2,\qquad A=\{0,1,q,q+m\}.
\end{equation}
Then $A$ is primitive,
\[
 \Delta_h(A)=M_{j_1}+M_{j_2},\qquad
 \operatorname{diam}(A)\le\tfrac12h^2+2h+\tfrac12<d_h.
\]
\end{lemma}
\begin{proof}
The set is increasing and primitive because $q>1$ and
it contains $0,1$.
For the relations, use the parametrization in
\cite[proof of Theorem 10.3, (189)--(190)]{Shin}.
For any homogeneous relation, set $s=z_2+z_3$ and $k=-z_3$.
The relation equation gives the unique parametrization
\begin{equation}\label{eq:two_relations}
 z=(e-s,-e,s+k,-k),\qquad e=qs-mk.
\end{equation}
For $s=0$, the primitive generator has degree $m+1=h-j_1$.
It is active exactly when $j_1\ge0$; its proper multiples
are inactive.

After changing sign, suppose $s\ge1$.
If $k<0$, then $e>q>h$, so the relation is not active.
For $k\ge0$, direct summation of the positive coordinates
in \eqref{eq:two_relations} gives
\[
 \rho(z)=s+k+\delta,\qquad
 \delta=\operatorname{dist}(e,[0,s]).
\]
Using $e=qs-mk$, we obtain
\[
 m\rho(z)-(q+m-1)s=m\delta-(e-s)\ge0.
\]
For $e\le s$, the last inequality follows from $\delta\ge0$.
For $e>s$, it is $(m-1)(e-s)\ge0$.
Thus every active relation with $s\ge1$ satisfies
\[
 (q+m-1)s\le mh.
\]
Write $p=m-2\ell\in\{0,1\}$.
Since $h-m=j_1+1$, substitution in \eqref{eq:two_family} yields
\[
 2(q+m-1)-mh=m(j_1+4+p)-p-2j_2-2>0.
\]
Indeed $m\ge h-4\ge13$, $j_1+4+p\ge3$, and
$p+2j_2+2\le9$.
It follows that $s=1$.

Let $L=h-\ell$.
At $k=L$, the value $e=\ell-j_2$ lies strictly between $1$
and $m$, and the degree is $h-j_2$.
At $k=L+1$, the degree is $h+p+j_2+2>h$.
At $k=L-1$, it is $h-j_2+m-1>h$.
As $k$ moves farther from $L$, the degree increases with
slope $m-1$ on the left and $m+1$ on the right.
There is therefore precisely one further active direction
when $j_2\ge0$, and none when $j_2=-1$.

There are at most two directions and no active multiples.
Since twice the smallest active degree exceeds $h$,
Lemma~\ref{lem:fiber_tools} applies, and no triple is possible.
Edge counting proves $\Delta_h(A)=M_{j_1}+M_{j_2}$, including
the case of no active relation.

Finally, $j_2\ge-1$ gives
\begin{align*}
 q+m
 &\le m(h+1)-(m-1)\ell+1\\
 &\le-\tfrac12m^2+(h+2)m+\tfrac12\\
 &\le\tfrac12h^2+2h+\tfrac12<d_h.
\end{align*}
The second inequality uses $\ell\ge(m-1)/2$.
For the third, the quadratic is increasing on $0\le m\le h$.
The last inequality holds for $h\ge17$.
\end{proof}

For three basic defects, we choose coprime elements $x,b$
so that the relation on $\{0,x,b\}$ has degree $b=r_3$.
We then choose the fourth element to give the two other
prescribed degrees. Taking $x$ close to $b/2$ keeps the
diameter near $3h^2/4$. The two possible triangles will
have height greater than $h$, so their directions
contribute edges without producing triple fibers.

\begin{lemma}[Three basic defects]\label{lem:three_defects}
Let $h\ge17$ and
$r_1,r_2,r_3\in\{h-3,h-2,h-1,h\}$.
Write $b=r_3$, and choose
\begin{equation}\label{eq:alpha_choice}
 \alpha=
 \begin{cases}
 (b-1)/2,&b\text{ odd},\\
 b/2-1,&b\equiv0\pmod4,\\
 b/2-2,&b\equiv2\pmod4.
 \end{cases}
\end{equation}
Put
\begin{equation}\label{eq:three_family}
 x=b-\alpha,\qquad \beta=r_1+\alpha-r_2,\qquad
 D=r_1b-\beta x,\qquad A=\{0,x,b,D\}.
\end{equation}
Then $A$ is primitive and
\[
 \Delta_h(A)=M_{h-r_1}+M_{h-r_2}+M_{h-r_3},\qquad
 D\le\tfrac34h^2+4<d_h.
\]
\end{lemma}
\begin{proof}
The choice \eqref{eq:alpha_choice} gives
\[
 b/2-2\le\alpha\le b/2,\qquad \gcd(\alpha,b)=1.
\]
For odd $b$, the gcd is immediate.
If $b\equiv0\pmod4$, it reduces to the gcd of an odd
number with two; if $b\equiv2\pmod4$, to the gcd of an
odd number with four.
Hence $0<x<b$ and $\gcd(x,b)=1$.

The identity
\begin{equation}\label{eq:D_convex}
 \frac Db=
 \frac{\alpha}{b}r_1+
 \left(1-\frac{\alpha}{b}\right)r_2-
 \alpha\left(1-\frac{\alpha}{b}\right)
\end{equation}
controls both the diameter and the possible relations.
The first two terms form a convex combination of numbers
in $[h-3,h]$. Moreover,
\[
 \frac b4-\frac4b
 \le\alpha(1-\alpha/b)\le\frac b4.
\]
Consequently
\begin{equation}\label{eq:three_bounds}
 D/b\ge3h/4-3>h/2,\qquad
 D\le hb-b^2/4+4\le3h^2/4+4<d_h.
\end{equation}
The upper bound uses that $hb-b^2/4$ is increasing for
$0\le b\le h$. In particular, $D>b$ and $A$ is a
primitive four-set.

To list the active relations, first note that
\begin{align}
 |Dz_3|
 &=|xz_1+bz_2|
   =|x(z_1+z_2)+\alpha z_2|\notag\\
 &\le x|z_1+z_2|+\alpha|z_2|
 \le(x+\alpha)h=bh. \label{eq:lastcoord}
\end{align}
The last inequality is the subset-mass bound
\eqref{eq:subsetmass}. Together with $D>bh/2$, it gives
$|z_3|\le1$.

For $z_3=0$, the primitive generator is
\[
 R=(-\alpha,b,-x,0),\qquad \rho(R)=b.
\]
Its proper multiples are inactive.
For $z_3=1$, every relation is $S+kR$, where
\begin{equation}\label{eq:three_S}
 S=(r_2-\alpha-1,\beta,-r_1,1).
\end{equation}
This follows from $\gcd(x,b)=1$ and the relation equation.
Since $b\ge h-3\ge14$ and $|r_1-r_2|\le3$,
$b/2-5\le\beta\le b/2+3$, so $0<\beta<b$.
Also $x\le h/2+2<h-3\le r_1$ and $r_2-\alpha-1\ge h/2-4>0$.
Thus
\[
 \rho(S)=r_1,\qquad
 S-R=(r_2-1,\beta-b,x-r_1,1),\qquad
 \rho(S-R)=r_2.
\]

For $k\ge1$, coordinate two of $S+kR$ is $-r_1-kx$.
Its absolute value is at least
\[
 r_1+x\ge(h-3)+(h-3)/2=\tfrac32h-\tfrac92>h.
\]
For $k\le-2$, coordinate one has absolute value at least
\[
 2b-\beta\ge\tfrac32(h-3)-3=\tfrac32h-\tfrac{15}2>h.
\]
Hence the complete active list is
$\pm R,\pm S,\pm(S-R)$.

The two triangle heights are
\[
 H=r_1+\alpha,\qquad H'=r_2+b-\alpha.
\]
Both exceed $h$: they are at least
$3h/2-13/2$ and $3h/2-6$, respectively.
There are no triple fibers.
Lemma~\ref{lem:fiber_tools} and the three relation degrees
now give the asserted defect.
\end{proof}

\section{The special defects and the upper bound}\label{sec:special}
The construction in Section~\ref{sec:family} realizes defect
seventeen at diameter $d_h$.
We now handle $23,26,29$, the other special entries of
Table~\ref{tab:exceptional}.
The previous construction kept both triangle orientations
above degree $h$. Here we bring one orientation down to
height $h$ or $h-1$, creating one or four triple fibers.
Subtracting these from the edge totals produces the
required defects.

\begin{lemma}\label{lem:special_three}
For every $h\ge17$, the defects $23,26,29$ have
integer four-set representatives of diameter at most $d_h$.
\end{lemma}
\begin{proof}
Most cases use \eqref{eq:three_family} with a smaller value
of $\alpha$. The parameters are listed in
Table~\ref{tab:special_params}.
\begin{table}[H]
\centering
\caption{Specializations of $\{0,b-\alpha,b,r_1b-\beta(b-\alpha)\}$,
where $\beta=r_1+\alpha-r_2$.}\label{tab:special_params}
\setlength{\tabcolsep}{4pt}
\begin{tabular}{ccccccc}
\toprule
Defect&Condition&$r_1$&$r_2$&$b$&$\alpha$&$D$\\
\midrule
23&$h$ even&$h-2$&$h-2$&$h-1$&2&$h^2-5h+8$\\
23&$h$ odd&$h-2$&$h-1$&$h-2$&2&$h^2-5h+8$\\
26&all $h$&$h-2$&$h-2$&$h-2$&1&$h^2-5h+7$\\
29&$h$ odd&$h-2$&$h-2$&$h-2$&2&$h^2-6h+12$\\
\bottomrule
\end{tabular}
\end{table}
Each set is primitive: either $\alpha=1$, or $\alpha=2$
and $b$ is odd. All displayed diameters are less than $d_h$.

The parametrization \eqref{eq:three_S} remains valid.
Here $D/b\ge h-4>h/2$, so \eqref{eq:lastcoord} again gives
$|z_3|\le1$.
For $k\ge1$, coordinate two of $S+kR$ has absolute value
at least $r_1+x\ge2h-6>h$.
For $k\le-2$, coordinate one has absolute value at least
$2b-\beta\ge2h-6>h$.
The complete active list is therefore
$\pm R,\pm S,\pm(S-R)$, with degrees $b,r_1,r_2$.
Their triangle heights are $H=r_1+\alpha$ and
$H'=r_2+b-\alpha$.
In the four rows of the table, $H$ is respectively
$h,h,h-1,h$, while $H'>h$.
There are respectively $1,1,4,1$ triple fibers.
The edge totals are $24,24,30,30$, so the defects are
$23,23,26,29$ as required.

It remains to realize defect twenty-nine when $h$ is even.
Put $n=h-4$ and $k=n/2$, and choose
\begin{equation}\label{eq:even_t}
 t=
 \begin{cases}
 k+1,&k\text{ even},\\
 k-2,&k\text{ odd}.
 \end{cases}
\end{equation}
Since even $h\ge17$ means $h\ge18$, we have $n\ge14$
and $1\le t\le n$.
Let
\begin{equation}\label{eq:even_29}
 D=n^2+2n+2,\qquad
 a=D-(n+1)t-n,\qquad b=D-nt+2,\qquad A=\{0,a,b,D\}.
\end{equation}
Then $a\ge2$, $b-a=t+n+2>0$, $b<D$, and
$D<d_h=n^2+4n+3$.

We check primitivity before using the relation vectors.
The gcd of the set entries is
\[
 \gcd(D,nt-2,(n+1)t+n)=\gcd(D,t+n+2).
\]
Indeed, subtracting the last two entries gives $t+n+2$.
Conversely, if a divisor of $D$ divides $t+n+2$, then
$t\equiv-n-2$ modulo that divisor, and
$nt-2\equiv-n^2-2n-2=-D$.
This proves the reverse divisibility.

For even $k$, $t+n+2=3(k+1)$; for odd $k$, it is $3k$.
Now $D=(2k+1)^2+1$ is not divisible by three.
In the first case,
$\gcd(D,k+1)=\gcd(2,k+1)=1$ because $k+1$ is odd.
In the second, $\gcd(D,k)=\gcd(2,k)=1$ because $k$ is odd.
Thus $A$ is primitive in every even order.

Consider
\[
 P=(1+t,1,n+2-t,0),\quad
 Q=(1,n+1,0,2),\quad R=(0,0,2,n+2).
\]
These exponent vectors refer to the coordinate order
$(0,b,D,a)$. Their degree is $h=n+4$.
The differences
\[
 u=Q-P=(-t,n,t-n-2,2),\qquad
 v=R-P=(-1-t,-1,t-n,n+2)
\]
have projected cross product $(-b,-D,-a)$.
By the primitivity just proved, they form a full basis
of the relation lattice.

For a relation $z=iu+jv$, coordinates one and three are
\[
 z_1=ni-j,\qquad z_3=2i+(n+2)j.
\]
The coefficient matrix has determinant $D$, so an active
relation satisfies
\[
 |i|\le\frac{(n+3)(n+4)}D<2,\qquad
 |j|\le\frac{(n+2)(n+4)}D<2.
\]
Both inequalities hold for $n\ge14$.
Thus $i,j\in\{-1,0,1\}$.
The vectors $u,v,u-v$ have degree $n+2=h-2$,
whereas $u+v$ has degree $2n+3>h$.
These three unsigned directions are the complete active list.

The common coordinate minimum of $P,Q,R$ is zero, so
their triangle has height $h$.
The opposite orientation has height $3(h-2)-h=2h-6>h$.
There is exactly one triple fiber.
The defect is therefore $3M_2-1=29$.
\end{proof}

\begin{theorem}[Uniform upper bound]\label{thm:uniform_upper}
For every integer $h\ge17$, $\nu(h,4)\le(h-3)(h-1)$.
\end{theorem}
\begin{proof}
The hypothesis $h\ge17$ ensures both $m\ge13$ in
Lemma~\ref{lem:two_defects} and $3h^2/4+4<d_h$ in
Lemma~\ref{lem:three_defects}. These estimates give a sufficient
threshold for the uniform construction.
Use Lemma~\ref{lem:replacement} with $D_0=d_h$.
Take a normalized representative $A$ with diameter greater
than $d_h$.
If its active rank is zero, or is one with relation degree
at least $h-3$, its defect is $0,1,4,10$ or $20$.
Lemma~\ref{lem:two_defects} provides a short representative
of the same cardinality.
If the active rank is one with smaller degree,
Lemma~\ref{lem:principal_representative} does so.

For active rank two,
Proposition~\ref{prop:large_defects} leaves sums of at most
three members of $\{1,4,10,20\}$ and the four defects
$17,23,26,29$.
Lemmas~\ref{lem:two_defects} and \ref{lem:three_defects}
handle the basic sums.
Proposition~\ref{thm:family} handles seventeen, and
Lemma~\ref{lem:special_three} handles the other special defects.
Every replacement has the same sumset cardinality and
diameter at most $d_h$, as required.
\end{proof}

Combining Theorems~\ref{thm:rigidity} and
\ref{thm:uniform_upper} proves
$\nu(h,4)=d_h$ for $h\ge17$.
The finite calculation in the next section completes
Theorem~\ref{thm:compression} for $2\le h\le16$.

\section{The finite calculation}\label{sec:computation}
\begin{theorem}[Computer-assisted]\label{thm:finite}
The values of $\nu(h,4)$ for $2\le h\le20$ are given in
Table~\ref{tab:finite}. In this range,
\[
 \nu(h,4)=D_h^{\max}\quad(2\le h\le10),\qquad
 \nu(h,4)=(h-3)(h-1)\quad(11\le h\le20).
\]
\end{theorem}

For $2\le h\le16$ we use the following finite coverage criterion;
the table also gives its values at $17\le h\le20$.

\begin{corollary}[Finite coverage criterion]\label{cor:coverage}
Let $C\ge h^2$, and enumerate all four-sets in $[0,C]\cap\Z$.
If their cardinalities include
\begin{equation}\label{eq:taillist}
 \{M_h\}\ \cup\
 \{M_h-M_{h-r}:2\le r\le h\},
\end{equation}
then the enumeration contains every element of $\calR(h,4)$.
The largest of the enumerated minimum realizing diameters is
$\nu(h,4)$.
\end{corollary}
\begin{proof}
Normalize an arbitrary realizing set. If its diameter is greater
than $C$, Lemma~\ref{lem:diameter} forces active rank at most one,
and Lemma~\ref{lem:rankone} places its cardinality in
\eqref{eq:taillist}. Those cardinalities are already present.
For each cardinality, the enumeration includes a representative of
its true minimum diameter: any smaller one would also be inside the
enumerated interval.
\end{proof}

The containment underlying this corollary is the quadratic core
statement of \cite[Corollary 10.4]{Shin}, which already places every
nonmaximal cardinality at normalized diameter at most $h^2$, and
which yields \eqref{eq:shinbounds}.
For the present calculation, the criterion is explicit:
the tail list \eqref{eq:taillist} must occur in the scan,
after which the largest recorded minimum is the exact
universal diameter.

For $2\le h\le20$, the search uses the cap
\[
 C_h=\max\{h^2,D_h^{\max}\},
\]
which is the upper bound in \eqref{eq:shinbounds}.
All sets \eqref{eq:normalized} are processed in increasing $D$.
Reflection about $D/2$ exchanges $(a,b)$ with $(D-b,D-a)$, so it
suffices to retain $a+b\le D$. Nonprimitive sets are omitted because
division by their common divisor gives a smaller representative of
the same cardinality.

For a fixed four-set, the set of attainable sums is stored as the
nonzero bit positions of an arbitrary-precision integer. Starting
with $B_0=1$, use
\begin{equation}\label{eq:bitset}
 B_{j+1}=B_j\mathbin{\vert}(B_j\ll a)
                  \mathbin{\vert}(B_j\ll b)
                  \mathbin{\vert}(B_j\ll D),
\end{equation}
where $\vert$ denotes bitwise union and $\ll$ is a left shift.
By induction, bit $t$ in $B_j$ is nonzero exactly when $t\in jA$.
The number of nonzero bits in $B_h$ is therefore $|hA|$.
The first diameter recorded for each cardinality is its minimum
within the cap.

\begin{table}[H]
\centering
\caption{Complete finite spectra and minimum universal diameters.}
\label{tab:finite}
\small
\begin{tabular}{rrrrr}
\toprule
$h$ & $C_h$ & $|\mathcal R(h,4)|$ & $D_h^{\max}$ & $\nu(h,4)$ \\
\midrule
2 & 6 & 4 & 6 & 6 \\
3 & 11 & 10 & 11 & 11 \\
4 & 16 & 17 & 15 & 15 \\
5 & 25 & 32 & 22 & 22 \\
6 & 36 & 51 & 28 & 28 \\
7 & 49 & 81 & 37 & 37 \\
8 & 64 & 113 & 45 & 45 \\
9 & 81 & 161 & 56 & 56 \\
10 & 100 & 213 & 66 & 66 \\
11 & 121 & 278 & 79 & 80 \\
12 & 144 & 350 & 91 & 99 \\
13 & 169 & 455 & 106 & 120 \\
14 & 196 & 551 & 120 & 143 \\
15 & 225 & 681 & 137 & 168 \\
16 & 256 & 808 & 153 & 195 \\
17 & 289 & 968 & 172 & 224 \\
18 & 324 & 1138 & 190 & 255 \\
19 & 361 & 1334 & 211 & 288 \\
20 & 400 & 1533 & 231 & 323 \\
\bottomrule
\end{tabular}
\end{table}

Every cardinality in \eqref{eq:taillist} occurred in the finite
search for each order. Corollary~\ref{cor:coverage} therefore
establishes global coverage, and Table~\ref{tab:finite} follows.
The values for $2\le h\le16$, together with the uniform
upper bound and defect-seventeen minimum for $h\ge17$,
complete the proof of Theorem~\ref{thm:compression}.
The first two entries $\nu(2,4)=6$ and $\nu(3,4)=11$ are already
stated in \cite[Corollary 10.5]{Shin}.
The ancillary files give the minima and one representative per cardinality.
Appendix~\ref{app:reproduction} gives the reproduction
commands and independent checks of the data.

\appendix
\section{Finite calculations}\label{app:reproduction}
For Section~\ref{sec:rigidity} the companion program
\texttt{defect17\_small.py} enumerates primitive four-sets through
diameter $h^2$. For the complete spectra in Section~\ref{sec:computation},
\texttt{python reproduce.py full} in the ancillary directory evaluates
the recurrence \eqref{eq:bitset} for the normalized sets through $C_h$
and checks the minima for $2\le h\le20$. The option \texttt{quick}
checks the rank-zero and rank-one tail cardinalities and the
minimum-diameter witnesses.

One comparison avoids both reflection and gcd reductions.
For all $\binom{79}{3}=79\,079$ triples $0<a<b<c\le79$, direct enumeration of
$\{ia+jb+kc:i,j,k\ge0,\ i+j+k\le11\}$ gives no set of size $347$.
For $\{0,7,17,80\}$ the $364$ multisets form $331$ singleton fibers,
$15$ double fibers and one triple fiber, giving $347$ distinct sums.

\section*{Use of generative AI}
Generative AI tools, including OpenAI Codex, were used in the mathematical
exploration, proof development, research programming, literature review
and drafting of this article. The author takes responsibility for the contents.

\end{document}